\documentclass[11pt,twoside]{amsart}
\usepackage{amssymb,amscd,color,amsthm}
\usepackage{tikz,tikz-cd}
\usetikzlibrary{calc, intersections, through}
\usepgflibrary{arrows}
\usetikzlibrary{positioning}
\usepackage{fancyhdr}
\usepackage{array}
\usepackage[all]{xy}
\usepackage{pdflscape}
\usepackage{afterpage}
\usepackage{capt-of}
\usepackage{enumitem}
\usepackage{lipsum}
\usepackage{graphicx,caption, subcaption}
\usepackage{bbm}
\usepackage[symbol]{footmisc}
\usepackage{xcolor}
\usepackage{comment}
\usepackage[square,numbers]{natbib}
\usetikzlibrary{arrows.meta}

\setlist{
  itemsep=2pt,
  topsep=3pt,
  parsep=0pt
}

\tikzstyle{wbullet}=[circle, draw=black, fill=white, thick, inner sep=0pt, minimum size=1.5mm]
\tikzstyle{wbullet}=[circle, draw=black, fill=white, thick, inner sep=0pt, minimum size=1.5mm]
\tikzstyle{bbullet}=[circle, draw=black, fill=black, inner sep=0pt, minimum size=1.5mm]
\tikzstyle{cross}=[circle, draw=black, fill=white, inner sep=0pt, minimum size=1.5mm]

\usepackage[
urlcolor=blue,
colorlinks=true,
linkcolor=blue,
citecolor=blue,
]{hyperref}

\usepackage{todonotes}

\theoremstyle{plain}
\newtheorem{theorem}{Theorem}[section]
\newtheorem{lemma}[theorem]{Lemma}
\newtheorem{corollary}[theorem]{Corollary}
\newtheorem{proposition}[theorem]{Proposition}
\theoremstyle{definition}
\newtheorem{definition}[theorem]{Definition}

\newtheorem{remark}[theorem]{Remark}
\newtheorem{claim}[theorem]{Claim}
\newtheorem{question}[theorem]{Question}

\newcommand{\bA}{\mathbb{A}}

\newcommand{\bC}{\mathbb{C}}

\newcommand{\bP}{\mathbb{P}}
\newcommand{\bQ}{\mathbb{Q}}
\newcommand{\bR}{\mathbb{R}}
\newcommand{\bZ}{\mathbb{Z}}

\newcommand{\cA}{\mathcal{A}}

\newcommand{\cC}{\mathcal{C}}

\newcommand{\cF}{\mathcal{F}}

\newcommand{\cM}{\mathcal{M}}
\newcommand{\cN}{\mathcal{N}}
\newcommand{\cO}{\mathcal{O}}
\newcommand{\cP}{\mathcal{P}}

\newcommand{\cS}{\mathcal{S}}
\newcommand{\cW}{\mathcal{W}}

\makeindex

\begin{document}
\title{On infinite accumulation points of log canonical volumes}
\author{Weili Shao}
\address{School of  Mathematical Sciences, Xiamen University, Xiamen, Fujian 361005, P.~R.~China}
\email{wlshaomath@stu.xmu.edu.cn}
\thanks{}
\date{\today}
\dedicatory{}

\begin{abstract}
    In this paper, we demonstrate a new phenomenon in the geography of volumes: for every integer $d\geq 2$, the set of volumes of $d$-dimensional projective log canonical varieties has an infinite accumulation point.
\end{abstract}

\maketitle

\pagestyle{myheadings}\markboth{\hfill \footnotesize  W. Shao\hfill}{\hfill \footnotesize On infinite accumulation points of log canonical volumes\hfill}

\tableofcontents

\section{Introduction}

We work over $\bC$ and use the standard notation and definitions of \cite{KM98,BCHM}.

The study of volumes is closely related to boundedness problems in birational geometry. Given a positive integer $d$ and a set $\cC\subset [0,1]$, set
\begin{itemize}
    \item $\cS(d,\cC)$: the set of projective log canonical pairs $(X, B)$ such that $\dim X=d$, $K_X+B$ is big and $B\in \cC$;
    \item $\mathrm{Vol}(d,\cC)$: the set of volumes $\mathrm{vol}(X, K_X+B)$ for pairs $(X, B)\in \cS(d,\cC)$. 
\end{itemize}

A fundamental result asserts that when $\cC$ satisfies the descending chain condition (DCC), the volume set $\mathrm{Vol}(d,\cC)$ also satisfies the DCC \cite{Alexeev94boundedness, HMX14}. 

When $\cC$ is finite, Alexeev and W.~Liu \cite[Theorem~1.1]{AlexeevLiu19accpoint} proved that in dimension two the accumulation points of $\mathrm{Vol}(2,\cC)$ arise precisely from accessible non-klt centres (see \cite[Definition~2.3]{AlexeevLiu19accpoint}); in higher dimensions, Filipazzi \cite[Theorem~1.3]{Filipazzi18volume} proved discreteness for volumes of $\epsilon$-log canonical pairs with coefficient set $\cC$.
These results reveal a close connection between the accumulation phenomena of log canonical volumes and the singularities of the underlying varieties.

In this paper, we focus on a stronger accumulation phenomenon, namely, points that remain accumulation points under arbitrarily many iterations. To make this precise, we introduce the notion of infinite accumulation point and recall the notion of accumulation complexity from \cite[Definitions 6.1 and 6.3]{AlexeevLiu19accpoint}.

\begin{definition}\label{Def1acc.comp.}
Let $\cA\subseteq \bR$ be a non-empty subset.
    \begin{itemize}
        \item[(1)] Let $\cA^{(0)}:=\cA$, and for $k\ge 1$ let $\cA^{(k)}$ be the set of accumulation points of $\cA^{(k-1)}$. If $a\in \cA^{(k)}$ for all $k\geq 1$, then $a$ is called an \textit{infinite accumulation point} of $\cA$ and we denote by $\cA^{(\infty)}$ the set of all infinite accumulation points of $\cA$. 
        \item[(2)] The set $\cA$ has \textit{finite accumulation complexity} if $\cA^{(k)}=\emptyset$ for $k\gg 0$. Otherwise, $\cA$ has \textit{infinite accumulation complexity}.
    \end{itemize}
\end{definition}

If a set has an infinite accumulation point, then it has infinite accumulation complexity. The converse does not hold in general. For example, \begin{align*}
\left\{\sum_{i=1}^{n}(1-\frac{1}{r_i})|r_i\in\bZ_{\geq 2}, n\in \bZ_{\geq 1}\right\}.
\end{align*}

In dimension two, Alexeev and W.~Liu \cite[Theorem 6.2]{AlexeevLiu19accpoint} showed that $\mathrm{Vol}(2,\{0\})$ has infinite accumulation complexity. Our main result shows that, in every dimension $d\ge 2$, the volume set actually admits an infinite accumulation point.

\begin{theorem}\label{Main-theorem}
For every integer $d\geq 2$ and every DCC set $\cC\subseteq [0,1]$, the set $\mathrm{Vol}(d, \cC)$ has an infinite accumulation point. 
\end{theorem}

This gives a negative answer to the question raised by Alexeev and W. Liu in \cite[Remark 6.4]{AlexeevLiu19accpoint}.

\begin{corollary}\label{Corollary}
    For every integer $d\geq 2$ and every DCC set $\cC\subseteq [0,1]$, there exists a positive real number \(M_d\), depending only on $d$, such that $\mathrm{Vol}(d, \cC)\cap (0,M_d]$ has infinite accumulation complexity.
\end{corollary}

The key geometric phenomenon occurs already in dimension two: the presence of a special non-klt centre gives rise to an infinite accumulation point.
For clarity, we state the result on the ample model.

\begin{theorem}\label{Main-theorem-geometric}
    Let $(X, B)$ be a $\bQ$-factorial projective log canonical surface with $K_X+B$ ample. Suppose that there exists a closed point $x$ on an irreducible component $B_0$ of $B^{=1}$ such that $x$ is a non-klt centre of $(X, B)$. Then
    \begin{align*}
        \mathrm{vol}(X, K_X+B)\in \mathrm{Vol}^{(\infty)}(2, \cC_B),
    \end{align*} where $\cC_B$ is the coefficient set of $B$.
\end{theorem}

We briefly outline the proof.

\noindent\textbf{Structure of the proof.}
We first prove a surface statement (Proposition~\ref{pre-main-theorem}) showing that if $(X,B)$ is a surface pair such that $B^{=1}$ contains two components $B_1,B_2$ intersecting transversally at a smooth point $x$ with $(K_X+B)\cdot B_1=0$ and $(K_X+B)\cdot B_2>0$, then $\mathrm{vol}(X,K_X+B)\in \mathrm{Vol}^{(\infty)}(2,\cC_B)$.

The proof of Proposition \ref{pre-main-theorem} proceeds as follows. Starting with the smooth germ $(X\ni x,B)$, we apply the toric modification induced by the subdivision of a smooth fan. This produces surfaces $W$ together with divisors $\bar B_W$ such that $\mathrm{vol}(W,K_W+\bar B_W)<\mathrm{vol}(X,K_X+B)$.
For a suitable subdivision, $K_W+\bar B_W$ is big and nef, and its volume equals $(K_W+\bar{B}_W)^2$. Contracting a component $C_0$ of $\bar B_W$ to preserve the coefficient set yields pairs $(Z,B_Z)\in \mathcal S(2,\cC_B)$ with the same volume $\mathrm{vol}(W,K_W+\bar B_W)$, and we show that the set of these volumes has $\mathrm{vol}(X,K_X+B)$ as an infinite accumulation point.

We then reduce Theorem~\ref{Main-theorem-geometric}
to Proposition~\ref{pre-main-theorem} by a birational modification. Theorem~\ref{Main-theorem} follows from
Theorem~\ref{Main-theorem-geometric} by taking the product with a fixed variety of general type and using the product formula for volumes.

\vspace{2pt}

\medskip
\noindent{\bf Acknowledgments.} 
The author is grateful to his advisor, Professor Wenfei Liu, for suggesting this problem, and for helpful discussions and advice. In particular, he explained the ideas of \cite{AlexeevLiu19accpoint} to the author. The revision was finished during the author's time as a visiting student at Kyoto University. The author would like to thank his mentor Professor Osamu Fujino for his hospitality. The author also thanks the anonymous referees for many helpful comments. This work is supported by National Natural Science Foundation of China (No. 12571046) and China Scholarship Council (No. 202506310055).

\medskip
\noindent\textbf{AI Declaration.}
ChatGPT 6 was used for mathematical discussions, proof checking, and language editing during the preparation of this revised version. All mathematical arguments were verified by the author.

\vspace{4pt}

\section{Toric construction}

\subsection{Preliminaries}

We recall some basic notions from toric geometry, see \cite{CoxLittleSchenck11Toricvarieties}.
Let $N\simeq \bZ^n$ be a lattice, and set $N_{\bR}:=N\otimes_{\bZ}\bR\simeq \bR^n$.

Throughout this paper, by a cone we mean a \textit{strongly convex rational polyhedral cone} $\sigma\subset N_\bR$, that is,
\begin{itemize}
    \item there are finitely many $u_1,u_2,\dots, u_d\in N$ such that
    \begin{align*}
        \sigma=\bR_{\geq 0}u_1+\dots +\bR_{\geq 0}u_d,
    \end{align*}
    \item and $\sigma\cap (-\sigma)=\{0\}$, i.e., $\sigma$ contains no lines.
\end{itemize}
For $u_1,u_2,\dots, u_d\in N$, we write $\mathrm{Cone}(u_1,\dots,u_d):=\bR_{\geq 0}u_1+\dots +\bR_{\geq 0}u_d$.

A one-dimensional face of a cone is called an \textit{edge}. 
A \textit{ray generator} of an edge is the first lattice point along the edge; in particular, it is a primitive lattice vector. 
By \cite[Lemma 1.2.15]{CoxLittleSchenck11Toricvarieties}, every cone is generated by the ray generators of its edges. These are called the \textit{minimal generators} of the cone. 
A cone is called \textit{simplicial} if its minimal generators are linearly independent over $\bR$. Note that in dimension two, every cone is simplicial.

A \textit{fan} $\Sigma$ is a collection of cones $\sigma\subset N_\bR$ satisfying
\begin{itemize}
    \item Each face $\tau$ of a cone $\sigma\in \Sigma$ belongs to $\Sigma$;
    \item The intersection of two cones in $\Sigma$ is a face of each.
\end{itemize}

Given a simplicial cone $\sigma\subset N_\bR$ with minimal generators $u_1,\dots,u_d$, we denote $N_\sigma$ for the lattice $(\bR u_1+\dots +\bR u_d)\cap N$. The \textit{multiplicity} of $\sigma$ is the index of the sublattice $\bZ u_1+\dots+\bZ u_d$ in $N_\sigma$,
\begin{align*}
    \mathrm{mult}(\sigma):=[N_\sigma:\bZ u_1+\dots +\bZ u_d].
\end{align*}
For basic properties of the multiplicity of a simplicial cone, see \cite[Proposition 11.4.8]{CoxLittleSchenck11Toricvarieties}.

From now on, we assume that $n=2$ and $N=\bZ e_1+\bZ e_2$, where $e_1,e_2$ is the standard basis of $\bR^2$. The following lemma computes the multiplicity.

\begin{lemma}\label{multiplicity-lemma}
    Let $\sigma \subseteq N_\mathbb{R}$ be a cone.
    \begin{itemize}
        \item[(1)] If $\sigma$ is one-dimensional with a minimal generator $u\in N$, then $\mathrm{mult}(\sigma) = 1$.
        \item[(2)] If $\sigma=\mathrm{Cone}(u_1,u_2)$ is two-dimensional, where $u_i = p_i e_1 + q_i e_2\in N$ with $p_i>0(i=1,2)$ and $0< \frac{q_1}{p_1}< \frac{q_2}{p_2}$, then
        \begin{align*}
            \mathrm{mult}(\sigma)=p_1q_2-p_2q_1.
        \end{align*}
    \end{itemize}
    \begin{proof}
        (1) Since $u$ is primitive, $N_\sigma=(\mathbb Ru)\cap N=\mathbb Zu$, hence $\mathrm{mult}(\sigma)=1$.

        (2) Since $\sigma$ is two-dimensional, $u_1$ and $u_2$ span $\bR^2$. Hence $N_\sigma=N$ and $e_1,e_2$ form a basis of $N_\sigma$. Let $\det(u_1,u_2)$ be the determinant of $u_1,u_2$ with respect to the basis $e_1,e_2$. By \cite[Proposition 11.4.8(c)]{CoxLittleSchenck11Toricvarieties}, 
        \begin{align*}
            \mathrm{mult}(\sigma)=[N:\mathbb Zu_1+\mathbb Zu_2]=|\det(u_1,u_2)|.
        \end{align*} 
        By assumption, $\det(u_1,u_2)>0$, and hence $\mathrm{mult}(\sigma)=p_1q_2-p_2q_1$.
    \end{proof}
\end{lemma}

Let $\Sigma$ be a two-dimensional fan and let $\Sigma(k)$ denote the set of $k$-dimensional cones of $\Sigma$. 
It determines a toric surface $X_\Sigma$. 
For $\tau \in \Sigma(k)$, denote by $V(\tau)$ the $(2-k)$-dimensional torus-invariant closed subvariety of $X_\Sigma$, and write $D_\rho := V(\rho)$ for $\rho \in \Sigma(1)$.

Let $\tau \in \Sigma(1)$ be a \textit{wall}, i.e., $\tau = \sigma \cap \sigma'$ for some $\sigma, \sigma' \in \Sigma(2)$. Write
\begin{align*}
\tau=\rho_2, \quad 
\sigma = \mathrm{Cone}(u_{\rho_1}, u_{\rho_2}), \quad
\sigma' = \mathrm{Cone}(u_{\rho_2}, u_{\rho_3}),
\end{align*}
where $u_{\rho_i}=p_ie_1+q_ie_2$ is the ray generator of $\rho_{i}$ for $i=1,2,3$.

The following proposition gives explicit formulas for the intersection numbers $D_\rho \cdot V(\tau)$.

\begin{proposition}\label{explicit-intersection}
Let $\tau=\sigma\cap\sigma'$ be a wall of $\Sigma$ as above, and let $\rho\in \Sigma(1)$. Suppose that $p_i>0$ for $i=1,2,3$ and $0<\frac{q_1}{p_1}<\frac{q_2}{p_2}<\frac{q_3}{p_3}$.
Then
\begin{itemize}
    \item[(1)] $D_{\rho}\cdot V(\tau)=0$ for all $\rho\notin\{\rho_1,\rho_2,\rho_3\}$.
    \item[(2)] $D_{\rho_1}\cdot V(\tau)=\dfrac{1}{p_1q_2-p_2q_1}$ and $D_{\rho_3}\cdot V(\tau)=\dfrac{1}{p_2q_3-p_3q_2}$.
    \item[(3)] $D_{\rho_2}\cdot V(\tau)=-\dfrac{p_1q_3-p_3q_1}{(p_1q_2-p_2q_1)(p_2q_3-p_3q_2)}$.
\end{itemize}
\begin{proof}
    By assumption, $\rho_1,\rho_2,\rho_3$ lie in the first quadrant and are ordered by slope. There exist three nonzero constants $a_1,a_2,a_3$ such that
    \begin{align*}
        a_1u_{\rho_1}+a_2u_{\rho_2}+a_3u_{\rho_3}=0.
    \end{align*}
    This relation is unique up to multiplication by a positive constant since $u_{\rho_1},u_{\rho_2}$ are linearly independent. Moreover,
    \begin{align}\label{identities}
        \frac{a_1}{a_2}=-\frac{p_2q_3-p_3q_2}{p_1q_3-p_3q_1}, \quad
        \frac{a_3}{a_2}=-\frac{p_1q_2-p_2q_1}{p_1q_3-p_3q_1}.
    \end{align}
    
    The result follows by substituting Lemma~\ref{multiplicity-lemma} and the identities in \eqref{identities} into \cite[Proposition~6.4.4]{CoxLittleSchenck11Toricvarieties}.
\end{proof}
\end{proposition}

\subsection{Construction of a series of toric surfaces}\label{toric-construction}

Consider the fan $\Sigma_0$ whose maximal cone is $\sigma=\mathrm{Cone}(e_1,e_2)$ and let $X=X_{\Sigma_0}\cong\mathbb A^2$ be the associated smooth toric variety. The torus-fixed point $x \in X$ corresponds to the maximal cone $\sigma$, while the torus-invariant divisors $B_1$ and $B_2$ correspond to the rays $\mathbb{R}_{\geq 0}e_1$ and $\mathbb{R}_{\geq 0}e_2$, respectively.

Consider a subdivision of $\sigma$ by adding rays $\bR_{\geq 0} u_k$, $1\leq k\leq n$, where
$u_k=p_ke_1+q_ke_2 \in N$ are vectors satisfying
\begin{itemize}
    \item $(p_k,q_k)\in \bZ^2_{\geq 1}$;
    \item $\mathrm{gcd}(p_k,q_k)=1$, i.e., the vector $u_k$ is primitive for $1\leq k\leq n$;
    \item $0<\frac{q_1}{p_1}<\frac{q_2}{p_2}<\cdots<\frac{q_n}{p_n}$.
\end{itemize}

Throughout this paper, we assume that every $(2n)$-tuple $(p_1,q_1,...,p_n,q_n)$ satisfies above three conditions.

Let $\Sigma$ be a fan whose maximal cones are $\sigma_k:=\mathrm{Cone}(u_{k-1}, u_{k})$, $1\leq k\leq n+1$. Here we set $u_0=e_1$ and $u_{n+1}=e_2$, i.e., $(p_0,q_0)=(1,0)$ and $(p_{n+1},q_{n+1})=(0,1)$. This subdivision of $\sigma$ defines a toric morphism $\pi\colon W=X_\Sigma\to X$. The following diagram illustrates this subdivision:

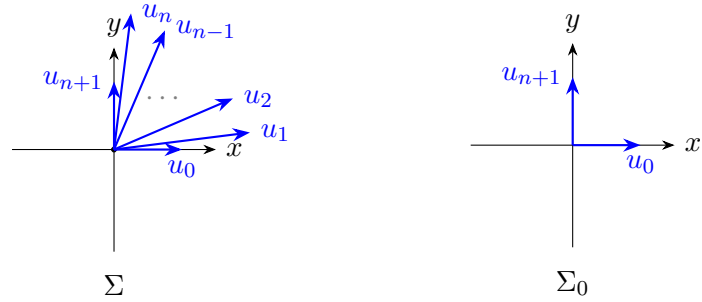
\begin{center}
\begin{tikzpicture}[scale=0.9, >=Stealth]

  \draw[->] (-1.5,0) -- (1.5,0) node[right]{$x$};
  \draw[->] (0,-1.5) -- (0,1.5) node[above]{$y$};

  \node[below] at (0,-1.7) {$\Sigma$};
  \filldraw (0,0) circle (1pt);
  \draw[->, thick, blue] (0,0) -- (1,0) node[below]{$u_0$};
  \draw[->, thick, blue] (0,0) -- (0,1) node[left]{$u_{n+1}$};
  \draw[->, thick, blue] (0,0) -- (2,0.25) node[right]{$u_{1}$};
  \draw[->, thick, blue] (0,0) -- (1.75,0.75) node[right]{$u_{2}$};
  \draw[->, thick, blue] (0,0) -- (0.25,2) node[right]{$u_{n}$};
  \draw[->, thick, blue] (0,0) -- (0.75,1.75) node[right]{$u_{n-1}$};
  \node[gray] at (0.75,0.75) {$\cdots$};

\end{tikzpicture}
\hspace{2cm}
\begin{tikzpicture}[scale=0.9, >=Stealth]
  \draw[->] (-1.5,0) -- (1.5,0) node[right]{$x$};
  \draw[->] (0,-1.5) -- (0,1.5) node[above]{$y$};

  \node[below] at (0,-1.7) {$\Sigma_0$};
  
  \draw[->, thick, blue] (0,0) -- (1,0) node[below]{$u_0$};
  \draw[->, thick, blue] (0,0) -- (0,1) node[left]{$u_{n+1}$};

\end{tikzpicture}
\end{center}

For the toric surface $W$, we denote by $C_k$ the torus-invariant curve corresponding to the ray $\mathbb{R}_{\geq 0}u_k$ for $0 \leq k \leq n+1$. In particular, $C_0$ and $C_{n+1}$ are the strict transforms of $B_1$ and $B_2$, respectively, and hence are not proper.

We then determine intersection numbers of torus-invariant curves on $W$.

\begin{lemma}\label{intersection-number}
    \begin{itemize}
        \item[(1)] If $|i-j|\geq 2$, then $C_i\cdot C_j=0$;
        
        \item[(2)] For $1 \leq k \leq n+1$,
        \begin{align*}
            C_{k-1}\cdot C_k=\frac{1}{p_{k-1}q_k-p_kq_{k-1}}.
        \end{align*}
        In particular, since $(p_0,q_0)=(1,0)$ and $(p_{n+1}, q_{n+1})=(0,1)$,
        \begin{align*}
            C_0\cdot C_1=\frac{1}{q_1}, \quad C_n\cdot C_{n+1}=\frac{1}{p_n}.
        \end{align*}
        
        \item[(3)] For $2\leq k\leq n-1$,
        \begin{align*}
            C_k^2=-\frac{p_{k-1}q_{k+1}-p_{k+1}q_{k-1}}{(p_{k-1}q_k-p_kq_{k-1})(p_kq_{k+1}-p_{k+1}q_k)}.
        \end{align*}
        For $k=1,n$, we have
        \begin{align*}
            C_1^2=-\frac{q_2}{q_1(p_1q_2-p_2q_1)}, \quad C_n^2=-\frac{p_{n-1}}{p_n(p_{n-1}q_n-p_nq_{n-1})}.
        \end{align*}
    \end{itemize}
    \begin{proof}
        Since $C_k=V(\mathrm{Cone}(u_k))=D_{\mathrm{Cone}(u_k)}$ for $0\leq k\leq n+1$, we may write
        \begin{align*}
            C_i\cdot C_j=D_{\mathrm{Cone}(u_i)}\cdot V(\mathrm{Cone}(u_j)).
        \end{align*}
        Applying Proposition~\ref{explicit-intersection} to the adjacent cones $\sigma_{k-1}, \sigma_{k}(1\leq k\leq n+1)$ gives the stated intersection numbers.
    \end{proof}
\end{lemma}

\section{Proof of the main results}

We begin by proving the following proposition, from which we deduce Theorem~\ref{Main-theorem-geometric}.

\begin{proposition}\label{pre-main-theorem}
    Let $(X, B)$ be a $\bQ$-factorial projective log canonical surface, with $K_X+B$ big and nef. Let $\cC_B$ be the coefficient set of $B$. Suppose that there exist two components $B_1,B_2$ of $B^{=1}$ and a smooth point $x$ such that
    \begin{enumerate}
        \item[(1)] $B_1$ and $B_2$ intersect transversally at $x$;
        \item[(2)] $B_1$ is the unique irreducible curve $C\subset X$ satisfying $(K_X+B)\cdot C=0$.
    \end{enumerate}
    Then $\mathrm{vol}(X, K_X+B)\in \mathrm{Vol}^{(\infty)}(2, \cC_B)$.
\end{proposition}

\subsection{Proof of Proposition \ref{pre-main-theorem}}

Throughout this section, we fix $(X,B), x$ as in Proposition~\ref{pre-main-theorem}. By the choice of $x$, $B_1$ and $B_2$ are determined uniquely.

The argument in this paragraph is inspired by \cite[5.12]{AM04}. Since $X$ is smooth at $x$ and $B_1,B_2$ intersect transversally at $x$,
after shrinking we may assume that there exists a Zariski open neighbourhood
$U$ of $x$ with regular parameters $u,v$
such that $B_1\cap U=(u=0)$ and $B_2\cap U=(v=0)$.
Then $(u,v)\colon (U\ni x)\to (\bA^2,0)$ is \'{e}tale at $x$.
Thus $(U,B_1\cap U+B_2\cap U)$ can be identified with
the affine toric surface $\bA^2$ associated to a fan
$\Sigma_0$ whose maximal cone is $\mathrm{Cone}(e_1,e_2)$.

Applying the toric construction in
Section~\ref{toric-construction} to $U$
yields a birational morphism
$\pi_U\colon W_U\to U$, which is an isomorphism over $U\setminus\{x\}$. 
Since the construction is local over $x$, we obtain a projective birational morphism 
$\pi\colon W \to X$ that is an isomorphism over $X\setminus\{x\}$.  We introduce the following divisors on $W$.
\begin{itemize}
    \item $B':=B-B_1-B_2$;
    \item $K_W+B_W:=\pi^*(K_X+B)=K_W+\pi_*^{-1}B'+C_0+C_{n+1}+\sum_{k=1}^n C_k$;
    \item $K_W+\tilde{B}_W:=K_W+\pi_*^{-1}B'+C_0+C_{n+1}$.
\end{itemize}

By Lemma~\ref{intersection-number},
\begin{align*}
    (K_W+\tilde{B}_W)\cdot C_0&=(K_W+B_W-\sum_{k=1}^n C_k)\cdot C_0=-C_0\cdot C_1<0.
\end{align*}
In particular, $C_0$ is $(K_W+\tilde{B}_W)$-negative. To study the contraction of $C_0$, we first compute its self-intersection number.

\begin{lemma}\label{C_0^2}
    $C_0^2=B_1^2-\dfrac{p_1}{q_1}$. 
    \begin{proof}
        By \cite[Lemma 5.13(2)]{AM04} (or standard toric geometry), since $u_k=p_ku_0+q_ku_{n+1}$ for $1\leq k\leq n$, we have
        \begin{align*}
            \pi^*(K_X+B_2)=K_W+C_{n+1}+(1-p_1)C_1+\cdots+(1-p_n)C_n,
        \end{align*}
        and
        \begin{align*}
            \pi^*(K_X+B_1+B_2)=K_W+C_0+C_{n+1}+C_1+\cdots+C_n.
        \end{align*}
        
        Therefore,
        \begin{align*}
            \pi^*B_1&=\pi^*(K_X+B_1+B_2)-\pi^*(K_X+B_2)\\
            &=C_0+p_1C_1+\cdots+p_nC_n.
        \end{align*}

        For $1\leq k\leq n$, $C_k$ is exceptional over $X$, it follows that $C_k\cdot \pi^*B_1=0$. By Lemma \ref{intersection-number}(1)(2), we obtain
        \begin{align*}
            B_1^2 = (\pi^*B_1 )^2=C_0\cdot \pi^*B_1= C_0^2 + p_1\, C_0 \cdot C_1= C_0^2 + \frac{p_1}{q_1}.
        \end{align*}
        Consequently, $C_0^2=B_1^2-\dfrac{p_1}{q_1}$. 
    \end{proof}
\end{lemma}

Since $K_X+B$ is big and nef and $(K_X+B)\cdot B_1=0$, the Hodge index theorem implies that $B_1^2<0$. By Lemma \ref{C_0^2},  $C_0^2<B_1^2<0$. Hence $\bR_{\geq 0}[C_0]$ is a $(K_W+\tilde{B}_W)$-negative extremal ray in $\overline{\mathrm{NE}}(W)$. By the contraction theorem (for example, \cite[Theorem 3.2]{Fujino12MMP}), there is a projective birational morphism $g\colon W\to Z$ contracting $C_0$ to a point. 

Define
\begin{itemize}
    \item $a_0:=\dfrac{(K_W+\tilde{B}_W)\cdot C_0}{C_0^2}=\dfrac{-C_0\cdot C_1}{C_0^2}>0$,
    \item $K_W+\bar{B}_W:=K_W+\tilde{B}_W-a_0C_0=K_W+B_W-\sum_{k=1}^{n}C_k-a_0C_0$.
\end{itemize}
By definition, $(K_W+\bar{B}_W)\cdot C_0=0$.

\begin{lemma}\label{a_0}
    $a_0=\dfrac{1}{p_1-q_1B_1^2}$. In particular, $a_0<\dfrac{1}{p_1}$.
    \begin{proof}
        This follows immediately from the definition of $a_0$ together with Lemma \ref{intersection-number}(2) and Lemma \ref{C_0^2}.
    \end{proof}
\end{lemma}

We now establish two lemmas which together provide a nef criterion for $K_W+\bar{B}_W$. Let $l=l(X, B, x)$ be the smallest positive integer such that
\begin{itemize}
    \item $K_X+\tilde{B}=K_X+B-\dfrac{1}{l}B_1-\dfrac{-B_1^2}{l(B_1\cdot B_2)}B_2$ is big and nef;
    \item $(K_X+B)\cdot B_2\geq \dfrac{B_1\cdot B_2}{l}$.
\end{itemize}

We first show that such a positive integer exists. By log abundance for log canonical surfaces (see \cite[Theorem 8.1]{Fujino12MMP}), $K_X+B$ is semiample. Since $K_X+B$ is big, it induces a projective birational morphism $h\colon X\to Y$. As $B_1$ is the unique irreducible curve $C\subset X$ satisfying $(K_X+B)\cdot C=0$, the exceptional locus of $h$ is precisely $B_1$.

Since $h$ contracts only $B_1$, we have $\rho(X/Y)=1$. Moreover,
\begin{align*}
    (K_X+B-\varepsilon B_2)\cdot B_1=-\varepsilon(B_1\cdot B_2)<0
\end{align*}
for $0<\varepsilon\ll 1$. Thus $h$ is a $(K_X+B-\varepsilon B_2)$-negative extremal contraction. It follows from \cite[Theorem 3.3]{Fujino12MMP} that $Y$ is $\bQ$-factorial. Set $B_Y:=h_*B$ and $B_{Y,2}:=h_*B_2$. Then $K_X+B=h^*(K_Y+B_Y)$, where $K_Y+B_Y$ is ample.

Since $K_Y+B_Y$ is ample and $(K_X+B)\cdot B_2>0$, we may choose a positive integer $m$ sufficiently large so that the following conditions hold:
\begin{itemize}
    \item $K_Y+B_Y-\dfrac{-B_1^2}{m(B_1\cdot B_2)}B_{Y,2}$ is ample;
    \item $(K_X+B)\cdot B_2=(K_Y+B_Y)\cdot B_{Y,2}\geq \dfrac{B_1\cdot B_2}{m}$.
\end{itemize}

Since $h$ contracts only $B_1$, we have $h^*B_{Y,2}=B_2+\dfrac{B_1\cdot B_2}{-B_1^2}B_1$. Consequently, the divisor
\begin{align*}
    K_X+B-\frac{1}{m}B_1-\frac{-B_1^2}{m(B_1\cdot B_2)}B_2=
    h^*\left(
    K_Y+B_Y-\frac{-B_1^2}{m(B_1\cdot B_2)}B_{Y,2}
    \right)
\end{align*}
is big and nef. Therefore, a positive integer satisfying the required conditions exists, and the smallest such integer $l$ is well-defined.

\begin{lemma}\label{nef1}
    Suppose that $(p_1,q_1,\cdots,p_n,q_n)$ satisfies the following system of inequalities:
    \begin{align*}
        p_k\geq l,~\text{for all } 1\leq k\leq n.
    \end{align*}
    Then 
    \begin{itemize}
        \item[(1)] $K_W+\bar{B}_W$ is big. In particular, $K_W+\tilde{B}_W$ is also big.
        \item[(2)] $(K_W+\bar{B}_W) \cdot C\geq 0$ for any irreducible curve $C \not\subseteq\mathrm{Supp}(\sum_{k=1}^{n} C_k)$.
    \end{itemize}
    \begin{proof}
        Let $E$ be a curve over $X$ associated to a ray $\bR_{\geq 0}u$, where $u=pe_1+qe_2$ is primitive with $p,q\in \bZ_{\geq 0}$. By the formula in \cite[Lemma 5.13(2)]{AM04}, we have
        \begin{align*}
            a(E, X, \tilde{B})=pa(B_1, X, \tilde{B})+qa(B_2, X, \tilde{B})=\frac{p}{l}+\frac{-qB_1^2}{l(B_1\cdot B_2)}.
        \end{align*}
        Since $p_k\geq l$ for all $1\leq k\leq n$, we have
        \begin{align*}
            a(C_k, X, \tilde{B})=\dfrac{p_k}{l}+\dfrac{-q_kB_1^2}{l(B_1\cdot B_2)}>\frac{p_k}{l}\geq 1=a(C_k, W, \bar{B}_W), 1\leq k\leq n.
        \end{align*}
        Similarly, we have
        \begin{align*}
            &a(C_{n+1}, X, \tilde{B})=\frac{p_{n+1}}{l}+\frac{-q_{n+1}B_1^2}{l(B_1\cdot B_2)}=\frac{-B_1^2}{l(B_1\cdot B_2)}>0=a(C_{n+1}, W, \bar{B}_W),\\
            &a(C_0, X, \tilde{B})=\frac{p_0}{l}+\frac{-q_0B_1^2}{l(B_1\cdot B_2)}=\frac{1}{l}\geq \frac{1}{p_1}> a_0=a(C_0, W, \bar{B}_W).
        \end{align*}
        Thus $K_W+\bar{B}_W-\pi^*(K_X+\tilde{B})$ is effective and is supported on $\mathrm{Supp}(\sum_{k=0}^{n+1} C_k)$. Consequently, $K_W+\bar{B}_W$ is big, and since $K_X+\tilde{B}$ is nef, $(K_W+\bar{B}_W) \cdot C\geq 0$ for any irreducible curve $C \not\subseteq\mathrm{Supp}(\sum_{k=0}^{n+1} C_k)$. By the definition of $K_W+\bar{B}_W$, $(K_W+\bar{B}_W)\cdot C_0=0$. 
        
        Since $B_1$ and $B_2$ intersect transversely at $x$ and $\pi$ separates their strict transforms over $x$, we have $C_0\cdot C_{n+1}=B_1\cdot B_2-1\geq 0$. Moreover, 
        \begin{align*}
            (K_W+\bar{B}_W)\cdot C_{n+1}&=(K_X+B)\cdot B_2-C_{n}\cdot C_{n+1}-a_0C_0\cdot C_{n+1}\\
            &\geq \frac{B_1\cdot B_2}{l}-\frac{1}{p_n}-a_0(B_1\cdot B_2-1)\\
            &=(\frac{1}{l}-\frac{1}{p_n})+(B_1\cdot B_2-1)\cdot (\frac{1}{l}-a_0)\geq 0.
        \end{align*}
        We complete the proof.
    \end{proof}
\end{lemma}

Lemma~\ref{nef1} deals with curves not contained in $\mathrm{Supp}(\sum_{k=1}^n C_k)$. We now treat curves supported on this locus.

\begin{lemma}\label{nef2}
    Suppose that $(p_1,q_1,\cdots,p_n,q_n)$ satisfies 
    \begin{align*}
        0<\frac{q_1}{p_1}<\frac{q_2}{p_2}<\cdots<\frac{q_n}{p_n},
    \end{align*}
    and the following system of inequalities:
    \begin{itemize}
        \item For $n=1$, $1-a_0 p_1\geq 0~($holds for any $p_1\geq 1)$.
        \item For $n=2$, 
            \begin{align*}
                p_1-p_2\geq 0,\quad q_2-q_1-a_0(p_1q_2&-p_2q_1)\geq 0.
            \end{align*}
        \item For $n\geq 3$,
            \begin{align*}
                p_{n-1}-p_n\geq 0,\quad q_2-q_1-a_0(p_1q_2&-p_2q_1)\geq 0,
            \end{align*}
            and for $2\leq k\leq n-1$,
            \begin{align*}
                (p_{k-1}q_{k+1}-p_{k+1}q_{k-1})-(p_{k-1}q_k-p_kq_{k-1})-(p_kq_{k+1}-p_{k+1}q_k) \geq 0.
            \end{align*}
    \end{itemize}
    Then $(K_W+\bar{B}_W) \cdot C\geq 0$ for any irreducible curve $C$ supported on $\mathrm{Supp}(\sum_{k=1}^{n} C_k)$.
    \begin{proof}
         This Lemma follows from Lemma \ref{intersection-number}. Note $K_W+B_W\equiv_X 0$ and $K_W+\bar{B}_W\equiv_X -a_0C_0-\sum_{k=1}^n C_k$, where we denote $\equiv_X$ for relative numerical equivalence over $X$. For $n=1$, 
         \begin{align*}
             (K_W+\bar{B}_W)\cdot C_1=(-a_0C_0-C_1)\cdot C_1=\frac{1-a_0p_1}{p_1q_1}\geq 0.
         \end{align*}
         For $n=2$,
         \begin{align*}
             (K_W+\bar{B}_W)\cdot C_1&=(-a_0C_0-C_1-C_2)\cdot C_1\\
             &=\frac{q_2-q_1-a_0(p_1q_2-p_2q_1)}{p_1(p_1q_2-p_2q_1)}\geq 0.
         \end{align*}
         and
         \begin{align*}
             (K_W+\bar{B}_W)\cdot C_2=(-C_1-C_2)\cdot C_2=\frac{p_1-p_2}{p_2(p_1q_2-p_2q_1)}\geq 0.
         \end{align*}
         Finally, we assume that $n\geq 3$. For $2\leq k\leq n-1$,
         \begin{align*}
             &(K_W+\bar{B}_W)\cdot C_k\\=&(-C_{k-1}-C_k-C_{k+1})\cdot C_k\\
             =&\frac{(p_{k-1}q_{k+1}-p_{k+1}q_{k-1})
             -(p_{k-1}q_{k}-p_{k}q_{k-1})-(p_{k}q_{k+1}-p_{k+1}q_k)}{(p_{k-1}q_{k}-p_{k}q_{k-1})(p_{k}q_{k+1}-p_{k+1}q_{k})}\geq 0.
         \end{align*}
         For $k=1,n$, we can use the similar argument as in the proof of $n=2$ case.
    \end{proof}
\end{lemma}

Define 
\begin{align*}
    K_Z+B_Z:=g_*(K_W+\tilde{B}_W)=g_*(K_W+\bar{B}_W).
\end{align*}
Since $W$ is $\bQ$-factorial, $(W,\tilde{B}_W)$ is log canonical, and $g$ is a $(K_W+\tilde{B}_W)$-negative extremal contraction, it follows from \cite[Theorem 3.3]{Fujino12MMP} that $Z$ is $\bQ$-factorial and $(Z,B_Z)$ is log canonical.

Under the assumptions of Lemmas \ref{nef1} and \ref{nef2}, $K_W+\bar{B}_W$ and $K_Z+B_Z$ are big and nef. Since $\tilde B_W$ has coefficients in $\cC_B$, the same holds for $B_Z$, and hence $(Z, B_Z)\in \cS(2,\cC_B)$. Since $g$ is $(K_W+\bar{B}_W)$-trivial, 
\begin{align*}
    K_W+\bar{B}_W=g^{*}(K_Z+B_Z).
\end{align*}
Therefore
\begin{align*}
    \mathrm{vol}(Z, K_Z+B_Z)=\mathrm{vol}(W, K_W+\bar{B}_W)=(K_W+\bar{B}_W)^2 \in \mathrm{Vol}(2,\cC_B).
\end{align*}

We now collect the above conditions into a parameter set.

\begin{definition}
    The set $\cS\cN\cP_n(X, B, x)$ consists of all $(2n)$-tuples $(p_1,q_1,
    \cdots,\\p_n,q_n)\in \bZ_{\geq 1}^{2n}$ satisfying the following conditions.
    \begin{itemize}
        \item $p_1>p_2>\cdots>p_n$ (monotonic condition);
        \item $0<\frac{q_1}{p_1}<\frac{q_2}{p_2}<\cdots<\frac{q_n}{p_n}$ (slope condition);
        \item inequalities in Lemmas \ref{nef1} and \ref{nef2} (nef condition);
        \item vector $u_k=p_ke_1+q_ke_2$ is primitive for $1\leq k\leq n$ (primitive condition).
    \end{itemize}
\end{definition}

Each element of $\cS\cN\cP_n(X, B, x)$ determines surface pairs $(W,\bar{B}_W)$ and $(Z,B_Z)$.

\begin{remark}
    The nef condition depends on the choice of the surface pair $(X, B)$ and the smooth point $x$, as the constants $l$ and $a_0$ in Lemmas~\ref{nef1} and \ref{nef2} depend on it. For simplicity, we write $\cS\cN\cP_n$ for $\cS\cN\cP_n(X, B, x)$ when $(X, B)$ and $x$ are clear from the context.
\end{remark}

We list some useful properties of $\cS\cN\cP_n$.

\begin{lemma}\label{keyLemma}
    \begin{itemize}
        \item[(1)] Suppose that $n\geq 2$. If $(p_1,q_1,\cdots,p_n,q_n)\in \cS\cN\cP_n$, then $(p_1,q_1,\cdots,p_{n-1},q_{n-1})\in \cS\cN\cP_{n-1}$.
        \item[(2)] Suppose that $n\geq 1$. If $(p_1,q_1,\cdots,p_n,q_n)\in \cS\cN\cP_n$, then for any integer $m\geq q_n$ with $\mathrm{gcd}(p_n,m)=1$, $(p_1,q_1,\cdots,p_n,m)\in \cS\cN\cP_n$.
        \item[(3)] Suppose that $n\geq 2$. If $(p_1,q_1,\cdots,p_{n-1}, q_{n-1})\in \cS\cN\cP_{n-1}$ and $p_{n-1}\geq l+1$, then there exist integers $p_n, q_n$ such that $(p_1,q_1,\cdots,p_n,q_n)\in \cS\cN\cP_n$. In particular, for any integer $m\geq q_n$ with $\mathrm{gcd}(p_n,m)=1$, $(p_1,q_1,\cdots,
        p_n,m)\in \cS\cN\cP_n$.
        \item[(4)] Suppose that $n\geq 1$. The set $\cS\cN\cP_n$ is nonempty.
    \end{itemize}
    \begin{proof}
        (1) and (2) are immediate from the definition of $\cS\cN\cP_n$.
        
        (3) We first consider the case $n=2$. Let $(p_1,q_1)\in \cS\cN\cP_1$ and $p_1\geq l+1$. Choose an integer $p_2$ such that $p_1>p_2\geq l$. Since $(p_1,q_1)$ is fixed, the constant $a_0$ is also fixed by Lemma \ref{a_0}. Choose $q_2$ satisfying $\mathrm{gcd}(p_2,q_2)=1$ and 
        \begin{align*}
            q_2\geq \max\left\{ \lceil \frac{p_2q_1}{p_1} \rceil,\lceil \frac{q_1-a_0p_2q_1}{1-a_0p_1} \rceil \right\}.
        \end{align*}
        The first inequality ensures slope condition and the second implies
        \begin{align*}
            q_2-q_1-a_0(p_1q_2&-p_2q_1)\geq 0.
        \end{align*}
        The remaining conditions are straightforward to verify, hence $(p_1,q_1,p_2,q_2)\in \cS\cN\cP_2$.
        
        Now assume $n\geq 3$ and let $(p_1,q_1,\cdots,p_{n-1},q_{n-1})\in \cS\cN\cP_{n-1}$ with $p_{n-1}\geq l+1$. Choose an integer $p_n$ such that $p_{n-1}>p_n\geq l$, and then choose $q_n$ satisfying $\mathrm{gcd}(p_n,q_n)=1$ and 
        \begin{align*}
            q_n\geq \max\left\{ \lceil \frac{p_nq_{n-1}}{p_{n-1}} \rceil, \lceil \frac{(p_{n-2}q_{n-1}-p_{n-1}q_{n-2})+p_n(q_{n-2}-q_{n-1})}{p_{n-2}-p_{n-1}} \rceil \right\}.
        \end{align*}
        The first inequality ensures slope condition and the second implies
        \begin{align*}
            (p_{n-2}q_{n}-p_{n}q_{n-2})-(p_{n-2}q_{n-1}-p_{n-1}q_{n-2})-(p_{n-1}q_{n}-p_{n}q_{n-1})\geq 0.
        \end{align*}
        The remaining conditions are straightforward to verify, hence $(p_1,q_1,\cdots,p_n,q_n)\in \cS\cN\cP_n$.
        
        (4) Choose $p_1\geq l+n-1$ and set $p_k=p_1-k+1$ for $2\leq k\leq n$. Let $q_1=1$, so that $(p_1,q_1)\in \cS\cN\cP_1$. Apply (3) inductively produces an element $(p_1,q_1,\cdots,p_n,q_n)\in \cS\cN\cP_n$.
    \end{proof}
\end{lemma}

We now express $\mathrm{vol}(Z,K_Z+B_Z)$ in terms of the parameters $(p_1,q_1,\dots,p_n,q_n)$. To simplify the notation, we introduce an auxiliary function $f_n=f_n^{(X, B, x)}$ on $\bZ^{2n}_{\geq 1}$. When $(X, B)$ and $x$ are clear, we simply write $f_n$. 

For $n\geq 3$, define
    \begin{align*}
        &f_n(p_1,q_1,\cdots,p_n,q_n)\\
        =&\left( \frac{q_2}{q_1}-1 \right) \cdot \frac{1}{p_1q_2-p_2q_1}+\left( \frac{p_{n-1}}{p_n}-1 \right) \cdot\frac{1}{p_{n-1}q_n-p_nq_{n-1}}\\
        &+\sum_{k=2}^{n-1} \frac{(p_{k-1}q_{k+1}-p_{k+1}q_{k-1})-(p_{k-1}q_k-p_kq_{k-1})-(p_{k}q_{k+1}-p_{k+1}q_k)}{(p_{k-1}q_k-p_kq_{k-1})(p_kq_{k+1}-p_{k+1}q_k)}\\
        &-\frac{1}{p_1-q_1B_1^2}\cdot \frac{1}{q_1}.
    \end{align*}
    
For $n=1,2$, define
    \begin{align*}
        f_2(p_1,q_1,p_2,q_2)=&\left( \frac{q_2}{q_1}-1 \right) \cdot \frac{1}{p_1q_2-p_2q_1}+\left( \frac{p_1}{p_2}-1 \right) \cdot \frac{1}{p_1q_2-p_2q_1}\\
        &-\frac{1}{p_1-q_1B_1^2}\cdot \frac{1}{q_1};\\
        f_1(p_1,q_1)=&\frac{1}{p_1q_1}-\frac{1}{p_1-q_1B_1^2}\cdot \frac{1}{q_1}.
    \end{align*}
    
We also set $f_0=0$. The following lemma shows that the function $f_n$ measures the drop of volume under the construction.

\begin{lemma}\label{relation-between-volume-F}
    Let $(p_1,q_1,\ldots,p_n,q_n)\in \cS\cN\cP_n$ and let $(W,\bar{B}_W)$ and $(Z,B_Z)$ be the associated pairs. Then
    \begin{align*}
        \mathrm{vol}(Z, K_Z+B_Z)=\mathrm{vol}(X, K_X+B)-f_n(p_1,q_1,\cdots,p_n,q_n).
    \end{align*}
    \begin{proof}
        \begin{align*}
            \mathrm{vol}(Z, K_Z+B_Z)&=(K_W+\bar{B}_W)^2\\
            &=(K_W+B_W)^2+(\sum_{k=1}^{n} C_k)^2+a_0C_0\cdot C_1\\
            &=\mathrm{vol}(X, K_X+B)+(\sum_{k=1}^{n} C_k)^2+a_0C_0\cdot C_1.
        \end{align*}
        By Lemma \ref{intersection-number},
        \begin{align*}
            (\sum_{k=1}^{n} C_k)^2+a_0C_0\cdot C_1=-f_n(p_1,q_1,\dots,p_n,q_n),
        \end{align*}
        the result follows.
    \end{proof}
\end{lemma}

Define
\begin{align*}
    \cF:=\{f_n(p_1,q_1,\cdots,p_n,q_n)|(p_1,\cdots,q_n)\in \cS\cN\cP_n, n\geq 1\}\cup \{0\}.
\end{align*}

The following lemma describes the iterated accumulation structure of $\cF$.

\begin{lemma}\label{DecreasingClaim}
    Let $(p_1,q_1,...,p_n,q_n)\in \cS\cN\cP_n$. Suppose that $s\geq 0$ and 
\begin{align*}
    f_n(p_1,q_1,\cdots,p_n,m)\in \cF^{(s)}
\end{align*}
for all integers $m\geq q_n$ with $\gcd(p_n,m)=1$. Then
\begin{align*}
    f_{n-1}(p_1,q_1,\cdots,p_{n-1},q_{n-1})\in \cF^{(s+1)}.
\end{align*}
    \begin{proof}
        By Lemma \ref{keyLemma}(2), 
        \begin{align*}
            (p_1,q_1,\cdots,p_n,m)\in \cS\cN\cP_n
        \end{align*}
        for all integers $m\geq q_n$ with $\gcd(p_n,m)=1$. Hence
        \begin{align*}
            \{f_n(p_1,q_1,\cdots,p_n,m)\}_{m\geq q_n,\gcd(p_n,m)=1}\subseteq \cF.
        \end{align*}
        It suffices to prove the following claim.
        \begin{claim}\label{claim-F}
            For fixed $(p_1,q_1,\cdots,p_n,q_n) \in \cS\cN\cP_n$, the sequence
            \begin{align*}
                \{f_n(p_1,q_1,\cdots,p_n,m)\}_{m\geq q_n,\gcd(p_n,m)=1}
            \end{align*}
            is strictly decreasing with
            \begin{align*}
                \lim_{m\to \infty} f_1(p_1,m)=f_0,
            \end{align*}
            and 
            \begin{align*}
                \lim_{m\to \infty} f_n(p_1,q_1,\cdots,p_n,m)=f_{n-1}(p_1,q_1,\cdots,p_{n-1},q_{n-1})
            \end{align*}
            for $n\geq 2$.
        \begin{proof}[Proof of the Claim \ref{claim-F}]
    For $n=1$, since
        \begin{align*}
            f_1(p_1,m)=\frac{1}{p_1m}-\frac{1}{p_1-mB_1^2}\cdot \frac{1}{m},
        \end{align*}
    the sequence $\{f_1(p_1,m)\}_{m\geq q_1, \gcd(p_1,m)=1}$ is strictly decreasing and converges to $0$.
    
    For $n=2$, 
    \begin{align*}
        f_2(p_1,q_1,p_2,m)=&\left( \frac{m}{q_1}+\frac{p_1}{p_2}-2\right) \cdot \frac{1}{p_1m-p_2q_1}-\frac{1}{p_1-q_1B_1^2}\cdot \frac{1}{q_1}\\
        =&\frac{1}{p_2q_1}\cdot \left( \frac{p_2}{p_1}+\frac{q_1}{p_1}\cdot \frac{(p_1-p_2)^2}{p_1m-p_2q_1} \right)-\frac{1}{p_1-q_1B_1^2}\cdot \frac{1}{q_1}\\
        =&f_1(p_1,q_1)+\frac{(p_1-p_2)^2}{p_1p_2}\cdot \frac{1}{p_1m-p_2q_1}.
    \end{align*}
    Since $p_1-p_2>0$, $(p_1-p_2)^2>0$ and the sequence $\{f_2(p_1,q_1,p_2,m)\}_{m\geq q_2, \gcd(p_2,m)=1}$ is strictly decreasing and converges to $f_1(p_1,q_1)$.

    For $n\geq 3$, 
        \begin{align*}
            f_n(p_1,q_1,\cdots,p_n,m)=&f_{n-1}(p_1,q_1,\cdots,p_{n-1},q_{n-1})+A_{(\vec{p},\vec{q})}(m),
        \end{align*}
        where $(\vec{p},\vec{q}):=(p_1,q_1,\dots, p_n,q_n)$ and
        \begin{align*}
            A_{(\vec{p},\vec{q})}(m):=&\frac{(p_{n-2}m-p_{n}q_{n-2})-(p_{n-2}q_{n-1}-p_{n-1}q_{n-2})-(p_{n-1}m-p_{n}q_{n-1})}{(p_{n-2}q_{n-1}-p_{n-1}q_{n-2})(p_{n-1}m-p_{n}q_{n-1})}\\
            &+\left( \frac{p_{n-1}}{p_n}-1 \right) \cdot \frac{1}{p_{n-1}m-p_nq_{n-1}}\\
            &-\left( \frac{p_{n-2}}{p_{n-1}}-1 \right) \cdot\frac{1}{p_{n-2}q_{n-1}-p_{n-1}q_{n-2}}.
        \end{align*}
        
        We compute
        \begin{align*}
            &\frac{(p_{n-2}m-p_{n}q_{n-2})-(p_{n-2}q_{n-1}-p_{n-1}q_{n-2})-(p_{n-1}m-p_{n}q_{n-1})}{(p_{n-2}q_{n-1}-p_{n-1}q_{n-2})(p_{n-1}m-p_{n}q_{n-1})}\\
            =&\frac{(p_{n-2}-p_{n-1})m-(p_{n-2}q_{n-1}-p_{n-1}q_{n-2})-p_n(q_{n-2}-q_{n-1})}{(p_{n-2}q_{n-1}-p_{n-1}q_{n-2})(p_{n-1}m-p_{n}q_{n-1})}\\
            =&\frac{1}{p_{n-2}q_{n-1}-p_{n-1}q_{n-2}}\cdot \left( \frac{p_{n-2}}{p_{n-1}}-1-\frac{(p_{n-1}-p_{n})(p_{n-2}q_{n-1}-p_{n-1}q_{n-2})}{p_{n-1}(p_{n-1}m-p_nq_{n-1})} \right)\\
            =&\left( \frac{p_{n-2}}{p_{n-1}}-1 \right) \cdot \frac{1}{p_{n-2}q_{n-1}-p_{n-1}q_{n-2}}+\left( \frac{p_n}{p_{n-1}}-1 \right) \cdot \frac{1}{p_{n-1}m-p_{n}q_{n-1}}.
        \end{align*}

        Therefore
        \begin{align*}
            A_{(\vec{p},\vec{q})}(m)=\left( \frac{p_{n-1}}{p_n}+\frac{p_n}{p_{n-1}}-2 \right)\cdot \frac{1}{p_{n-1}m-p_{n}q_{n-1}}.
        \end{align*}

        Since
        \begin{align*}
            \frac{p_{n-1}}{p_n}+\frac{p_n}{p_{n-1}}-2>0
        \end{align*}
        and
        \begin{align*}
            p_{n-1}m-p_nq_{n-1}\ge p_{n-1}q_n-p_nq_{n-1}>0,
        \end{align*}
        the sequence $\{A_{(\vec{p},\vec{q})}(m)\}_{m\ge q_n,\gcd(p_n,m)=1}$ is strictly decreasing with limit $0$. Hence 
        \begin{align*}
            \lim_{m\to \infty} f_n(p_1,q_1,\cdots,p_n,m)=f_{n-1}(p_1,q_1,\cdots,p_{n-1},q_{n-1}).
        \end{align*}
        This proves the claim and hence the lemma.
            \end{proof}
        \end{claim}
    \end{proof}
\end{lemma}

\begin{proposition}\label{fuzhu}
    The set $\cF$ has an infinite accumulation point.
    \begin{proof}
        It suffices to show that $0\in\cF^{(n)}$ for all $n\ge1$.
        
        Fix $n\geq 1$ and set $\vec{p}:=(p_1,p_2,\cdots,p_n)=(l+n-1,l+n-2,\cdots,l)$. 
        Since $(l+n-1,1)\in\cS\cN\cP_1$, we set $q_1(\vec{p}):=1$. We define $\cM_1$ for the set of all integers $m_1$ such that $m_1\geq q_1(\vec{p})$ and $\gcd(l+n-1,m_1)=1$.
        By Lemma \ref{keyLemma}(2), 
        \begin{align*}
            m_1\in \cM_1\implies(l+n-1,m_1)\in \cS\cN\cP_1.
        \end{align*}

        Assume now that $n\geq 2$. We will define $\cM_k$ for $2\leq k\leq n$ by induction. Suppose that $\cM_{k-1}$ has been defined. For each $(m_1,\dots,m_{k-1})\in \cM_{k-1}$, By Lemma \ref{keyLemma}(3), let $q_k(\vec{p};m_1,\dots,m_{k-1})$ be the smallest integer such that 
        \begin{align*}
            (l+n-1,m_1,\ldots,l+n-k,q_k(\vec{p};m_1,\dots,m_{k-1}))\in\cS\cN\cP_k.
        \end{align*}
        We define $\cM_{k}$ for the set of all $k$-tuples $(m_1,\dots,m_{k-1},m_k)$ such that
        \begin{itemize}
            \item $(m_1,\dots,m_{k-1})\in \cM_{k-1}$;
            \item $m_k\geq q_k(\vec{p};m_1,\dots,m_{k-1})$ and $\gcd(l+n-k,m_k)=1$.
        \end{itemize}
        By Lemma \ref{keyLemma}(2),
        \begin{align*}
            (m_1,\ldots,m_k)\in\cM_k \Longrightarrow (l+n-1,m_1,\ldots,l+n-k,m_k)\in\cS\cN\cP_k.
        \end{align*}
        Note that if $1\le j<k\le n$ and $(m_1,\ldots,m_k)\in\cM_k$, then $(m_1,\ldots,m_j)\in\cM_j$.
        
        By the definition of $\cM_n$ and $\cF$, for every $(m_1,\cdots,m_{n-1},m_n)\in \cM_n$, we have
        \begin{align*}
            f_n(l+n-1,m_1,\cdots,l,m_n)\in \cF.
        \end{align*}

        We now apply Lemma~\ref{DecreasingClaim} iteratively.
    \begin{itemize}
        \item Apply it to $(l+n-1,m_1,\cdots, l,m_n)$ for all $(m_1,\cdots, m_n)\in \cM_n$, we obtain
            \begin{align*}
                f_{n-1}(l+n-1,m_1,\cdots,l+1,m_{n-1})\in \cF^{(1)},
            \end{align*}
        for all $(m_1,\cdots,m_{n-2},m_{n-1})\in \cM_{n-1}$;
        \item After $n-2$ further applications, we obtain
            \begin{align*}
                f_{1}(l+n-1,m_1)\in \cF^{(n-1)},
            \end{align*}
        for all $m_1\in \cM_{1}$;
        \item A final application yields $0=f_0\in \cF^{(n)}$.
    \end{itemize}
    
    Since $n\geq 1$ is arbitrary, we conclude that $0\in \cF^{(n)}$ for all $n\geq 1$. Hence, $\cF$ has an infinite accumulation point.
    \end{proof}
\end{proposition}

We are now ready to prove Proposition \ref{pre-main-theorem}.

\begin{proof}[Proof of Proposition \ref{pre-main-theorem}]
    We define a subset of $\mathrm{Vol}(2,\cC_B)$ by
    \begin{align*}
        \cW:=&\{\mathrm{vol}(Z, K_Z+B_Z)| (Z, B_Z) \mathrm{~corresponds~to~} (p_1,\cdots,q_n)\in \cS\cN\cP_n\}\\
        &\cup \{\mathrm{vol}(X, K_X+B)\}.
    \end{align*}
    
    There is an isometric (with respect to the Euclidean metric on $\bR$) bijection between $\cF$ and $\cW$. In particular, this bijection is also a bijection between $\cF^{(n)}$ and $\cW^{(n)}$ for all $n\geq 1$.

    By Lemma~\ref{fuzhu}, $\cF$ has an infinite accumulation point at $0$. Therefore $\cW$ has an infinite accumulation point at $\mathrm{vol}(X,K_X+B)$, and hence $\mathrm{vol}(X,K_X+B)$ is an infinite accumulation point of $\mathrm{Vol}(2,\cC_B)$.
\end{proof}

\subsection{Proof of Theorems \ref{Main-theorem},~\ref{Main-theorem-geometric} and Corollary \ref{Corollary}}

\begin{proof}[Proof of Theorem \ref{Main-theorem-geometric}]
    Suppose that $(X,B), B_0$ and $x$ satisfy conditions of Theorem \ref{Main-theorem-geometric}. By the LMMP for surfaces (for example, \cite[Theorem 1.1]{Fujino12MMP}), there exists a birational morphism $f\colon Y\to X$ such that
        \begin{itemize}
            \item $Y$ is $\bQ$-factorial;
            \item $K_Y+B_Y:=f^*(K_X+B)$ is big and nef;
            \item $E\subset Y$ is the only exceptional divisor of $f$. It satisfies that $a(E,X,B)=0$ and $\mathrm{center}_X(E)=x$;
            \item $f^{-1}_*B_0$ and $E$ intersect transversally at a smooth point $y\in Y$.
        \end{itemize}
        Thus $(Y, B_Y),f^{-1}_*B_0, E, y$ satisfy the conditions of Proposition \ref{pre-main-theorem} and $B_Y\in \cC_B$. By Proposition \ref{pre-main-theorem}, 
        \begin{align*}
            \mathrm{vol}(X, K_X+B)=\mathrm{vol}(Y, K_Y+B_Y)\in \mathrm{Vol}^{(\infty)}(2, \cC_B).
        \end{align*}
        This complete the proof.
\end{proof}

\begin{proof}[Proof of Theorem \ref{Main-theorem}]
    Since $\mathrm{Vol}(d, \{0\})$ is a subset of $\mathrm{Vol}(d,\cC)$ for any DCC set $\cC$, it suffices to show that the set $\mathrm{Vol}^{(\infty)}(d, \{0\})$ is non-empty for $d\geq 2$.
    
    We first treat the case $d=2$.
    Consider the surface pair $(\bP^2, \sum_{k=1}^4 L_k)$, where the $L_k$ are four distinct lines in general position. It satisfies the conditions of Theorem \ref{Main-theorem-geometric} and hence $1\in \mathrm{Vol}^{(\infty)}(2, \{0,1\})$. By \cite[Theorem 1.1]{AlexeevLiu19accpoint}, 
    \begin{align*}
        \mathrm{Vol}^{(1)}(2, \{0\})=\mathrm{Vol}^{(1)}(2, \{0,1\}).
    \end{align*}
    Therefore $1\in \mathrm{Vol}^{(\infty)}(2, \{0\})$.

    Now we assume $d\geq 3$. Let $X_{d+1}$ be a general hypersurface of degree $d+1$ in $\bP^{d-1}$. Then $\dim X_{d+1}=d-2$ and $K_{X_{d+1}}=\cO_{X_{d+1}}(1)$ is ample. 

    For any surface $Z\in\cS(2,\{0\})$, the product $Z\times X_{d+1}$ is a $d$-dimensional log canonical variety of log general type. By the product formula for volumes,
    \begin{align*}
        \mathrm{vol}(Z\times X_{d+1},K_{Z\times X_{d+1}})=\binom{d}{2}\cdot \mathrm{vol}(Z,K_Z)\cdot \mathrm{vol}(X_{d+1},K_{X_{d+1}}).
    \end{align*}
    Since $\mathrm{vol}(X_{d+1},K_{X_{d+1}})=d+1$, we obtain 
    \begin{align*}
        \mathrm{vol}(Z\times X_{d+1},K_{Z\times X_{d+1}})=\frac{d(d-1)(d+1)}{2}\cdot \mathrm{vol}(Z,K_Z).
    \end{align*}

    Hence
    \begin{align*}
        \frac{d(d-1)(d+1)}{2}\cdot\mathrm{Vol}(2,\{0\}) \subseteq \mathrm{Vol}(d,\{0\}).
    \end{align*}
    Since $1\in\mathrm{Vol}^{(\infty)}(2,\{0\})$, we have $\frac{d(d-1)(d+1)}{2}\in \mathrm{Vol}^{(\infty)}(d, \{0\})$. Hence $\mathrm{Vol}^{(\infty)}(d,\{0\})$ is nonempty as required.
\end{proof}

\begin{proof}[Proof of Corollary \ref{Corollary}]
    By Theorem \ref{Main-theorem}, there exists $M_d\in \mathrm{Vol}^{(\infty)}(d,\{0\})$ such that the set $\mathrm{Vol}^{(\infty)}(d,\{0\})\cap (0,M_d]$ is nonempty. In particular, $\mathrm{Vol}^{(n)}(d,\{0\})\cap (0,M_d]$ is non-empty for any $n\geq 1$.
    
    Since $\mathrm{Vol}(d,\{0\})$ is a DCC set, for any $n\geq 1$,
    \begin{align*}
        (\mathrm{Vol}(d,\{0\}) \cap (0,M_d])^{(n)}=\mathrm{Vol}^{(n)}(d,\{0\}) \cap (0,M_d].
    \end{align*}
    This can be proved by induction on $n$, using the DCC property in the case $n=1$, so we omit the proof.

    Thus the set $(\mathrm{Vol}(d, \{0\})\cap (0,M_d])^{(n)}$ is nonempty for any $n\geq 1$. Therefore $\mathrm{Vol}(d,\{0\})\cap (0,M_d]$ has infinite accumulation complexity and so does $\mathrm{Vol}(d, \cC)\cap (0,M_d]$ for any DCC set $\cC\subseteq [0,1]$.
\end{proof}

\section{Further discussions}

As an application, we show that an analogue of
\cite[Theorem~1.4(1)]{chenhanliu2024iitakavolumeslogcanonical} fails for Iitaka dimension $\kappa\ge2$. 
Here $\mathrm{Ivol}^{\cC}_{\mathrm{lc}}(d,\kappa)$
denotes the set of Iitaka volumes of $d$-dimensional
log canonical pairs with Iitaka dimension $\kappa$
and coefficients in $\cC$
(see \cite[Introduction]{chenhanliu2024iitakavolumeslogcanonical}).

\begin{corollary}\label{CorIitaka}
    For integers $d$ and $\kappa$ such that $2\leq \kappa\leq d$, $\mathrm{Ivol}^{\{0\}}_{\mathrm{lc}}(d,\kappa)$ has an infinite accumulation point. 
    \begin{proof}
        The $\kappa=d$ case follows from Theorem \ref{Main-theorem}. 
    
    Now assume $2\leq \kappa<d$. Let $X$ be a $\kappa$-dimensional log canonical variety with $K_X$ ample, and let $E$ be a $(d-\kappa)$-dimensional smooth Calabi-Yau variety. Consider the product $X\times E$.

    Since $K_E\sim 0$, we have $K_{X\times E}=p_X^*K_X$, where $p_X$ is the projection onto $X$. Hence $\kappa(X\times E)=\kappa$ and
    \begin{align*}
        \mathrm{vol}_\kappa(X\times E,K_{X\times E})=\mathrm{vol}(X,K_X),
    \end{align*}
    where $\mathrm{vol}_\kappa$ denotes the Iitaka volume
(see \cite[Introduction]{chenhanliu2024iitakavolumeslogcanonical}).

    Therefore
    \begin{align*}
        \mathrm{Vol}(\kappa,\{0\})\subseteq \mathrm{Ivol}^{\{0\}}_{\mathrm{lc}}(d,\kappa).
    \end{align*}
    Since $\mathrm{Vol}(\kappa,\{0\})$ has an infinite accumulation point, so does $\mathrm{Ivol}^{\{0\}}_{\mathrm{lc}}(d,\kappa)$.
    \end{proof}
\end{corollary}

It is natural to ask whether the presence of a non-klt locus described in Theorem~\ref{Main-theorem-geometric} is also necessary for a volume to be an infinite accumulation point. 

\begin{question}
    Let $v_\infty$ be an infinite accumulation point of $\mathrm{Vol}(2,\{0\})$. Does there exist a log canonical surface $(X,B)\in\cS(2,\{0,1\})$ such that
    \begin{enumerate}
        \item $K_X+B$ is ample and $v_\infty=\mathrm{vol}(X, K_X+B)$;
        \item some irreducible component of $B^{=1}$ contains a zero-dimensional non-klt centre of $(X,B)$?
    \end{enumerate}
\end{question}

We pose the following question concerning the explicit birational geometry.

\begin{question}
    What is the minimal infinite accumulation point $m_2$ of $\mathrm{Vol}(2,\{0\})$? 
\end{question}

By \cite[Theorem 1.1]{LiuLiu25minimal}, the minimal accumulation point of $\mathrm{Vol}(2, \{0\})$ is $\frac{1}{825}$. Therefore, $m_2>\frac{1}{825}$.

On the other hand, by \cite[Theorem 1.4]{LiuShokurov2023optimal}, the minimal volume of log canonical surface pairs $(X,B)$ such that $K_X+B$ is ample and $B$ is a nonzero reduced divisor is $\frac{1}{462}$. Moreover, \cite[Lemma 3.5]{LiuLiu25minimal} shows that in the extremal case $(X,B)$ is plt and no closed point on $B$ is a non-klt centre of $(X,B)$. Hence such pairs do not satisfy the conditions in Theorem \ref{Main-theorem-geometric}.

Motivated by \cite[Lemma 3.5]{LiuLiu25minimal}, we expect that $m_2>\frac{1}{462}$.

\bibliographystyle{alpha}
\bibliography{ref/Dissertation_full}  

\end{document}